\documentclass[12pt]{amsart}
\usepackage{geometry}   %????????
\usepackage[colorlinks,citecolor = red, linkcolor=blue,hyperindex]{hyperref}
\usepackage{euscript,eufrak,verbatim, mathrsfs, amscd}
\usepackage[psamsfonts]{amssymb}
\usepackage{bbm}
\usepackage{graphicx}
 \usepackage{float}
\usepackage{float, tikz}

 \usepackage[all, cmtip]{xy}

\usepackage{upref, xcolor, dsfont}
\usepackage{amsfonts,amsmath,amstext,amsbsy, amsopn,amsthm}
\usepackage{enumerate}

\usepackage{url}

\usepackage{mathtools}
\usepackage{bookmark}

 \usepackage{euscript}
\usepackage{helvet}         % selects\textbf{\textbf{•}} Helvetica as sans-serif font
\usepackage{courier}        % selects Courier as typewriter font
\usepackage{type1cm}        % activate if the above 3 fonts are
\usepackage{multicol}        % used for the two-column index
\usepackage[bottom]{footmisc}% places footnotes at page bottom

\newtheorem{theorem}{Theorem}[section]
\newtheorem*{theorem*}{Theorem B}
\newtheorem{lemma}[theorem]{Lemma}

\newtheorem{proposition}[theorem]{Proposition}

\newtheorem*{definition*}{Definition}
\newtheorem*{remark*}{Remark}

\newtheorem*{observation*}{Observation}

\newtheorem*{assumption*}{Assumption}

\theoremstyle{definition}

\theoremstyle{remark}

\begin{document}

\title[Toeplitz operators]{Toeplitz operators between  distinct Bergman spaces}

%\author%[authorlabel1]
%{Yongjiang Duan}
%\address%[authorlabel1]
%{Yongjiang DUAN: School of
%Mathematics and Statistics,  Northeast Normal University, Changchun, Jilin 130024, P.R.China}
%\email{duanyj086@nenu.edu.cn}
%\email{tba}

%\author%[authorlabel1]
%{Kunyu Guo}
%\address%[authorlabel1]
%{Kunyu GUO: School of Mathematical Sciences, Fudan University, Shanghai 200433, P.R.China}
%\email{kyguo@fudan.edu.cn}

\author%[authorlabel1]
{Siyu Wang}
\address%[authorlabel1]
{Siyu WANG: School of
Mathematics and Statistics, Northeast Normal University, Changchun, Jilin 130024, P.R.China}
\email{wangsy696@nenu.edu.cn}

\author{Zipeng Wang}
\address{Zipeng WANG: College of Mathematics and Statistics, Chongqing University, Chongqing,
401331, P.R.China}
\email{zipengwang2012@gmail.com, zipengwang@cqu.edu.cn}
\begin{abstract}
For $-1<\alpha<\infty$, let $\omega_\alpha(z)=(1+\alpha)(1-|z|^2)^\alpha$ be the standard weight on the unit disk. In this note, we provide descriptions of the boundedness and compactness for the Toeplitz operators $T_{\mu,\beta}$ between distinct weighted Bergman spaces $L_{a}^{p}(\omega_{\alpha})$ and $L_{a}^{q}(\omega_{\beta})$ when $0<p\leq1$, $q=1$, $-1<\alpha,\beta<\infty$ and $0<p\leq 1<q<\infty, -1<\beta\leq\alpha<\infty$, respectively. Our results can be viewed as extensions of Pau and Zhao's work in \cite{Pau}. Moreover, partial of main results are new even in the unweighted settings.
\end{abstract}

\subjclass[2010]{47B35; 30H20}
\keywords{Toeplitz operator, Carleson measures, distinct Bergman spaces}

\maketitle

\section{Introduction}
For $0<p<+\infty$ and $-1<\alpha<+\infty$, let $L_a^p(\omega_\alpha)$ be the \textit{weighted Bergman space} which contains analytic functions $f$ on $\mathbb{D}$ such that
\[
\|f\|_{L^p(\omega_\alpha)}:=\bigg(\int_{\mathbb{D}}|f(z)|^p\omega_\alpha(z)dA(z)\bigg)^\frac{1}{p}<+\infty,
\]
where $dA$ is the normalized area measure on $\mathbb{D}$ and $\omega_\alpha(z)=(1+\alpha)(1-|z|^2)^\alpha$ is the standard weight. For a finite positive Borel measure $\mu$ on $\mathbb{D}$, define the \textit{Toeplitz operator} $T_{\mu,\alpha}$ by
\[
T_{\mu,\alpha}(f)(z)=\int_{\mathbb{D}}\frac{f(w)}{(1-z\bar{w})^{2+\alpha}}d\mu(w),~z\in\mathbb{D}.
\]

\vskip 0.1in

The boundedness and compactness of the Toeplitz operators $T_{\mu,\alpha}$ on the Bergman space $L^p_a(\omega_\alpha)$ are classical in operator theory on Bergman spaces \cite{zhu2007}.
When $1<p<+\infty$, these problems were solved by Luecking and Zhu \cite{L1987, Zhu1988}, while the case $p=1$ were considered in \cite{WL-2010, WZZ-2006}. In \cite{PR,PRS}, Pel\'{a}ez, R\"{a}tty\"{a} and Sierra completely characterized bounded and compact Toeplitz operators between distinct weighted Bergman spaces $L^p_a(\omega)$ and $L^q_a(\omega)$ for $1<p,q<+\infty$ if the weight $\omega$ is regular. Very recently, Duan, Guo and the current authors \cite{DGWW} extended Pel\'{a}ez, R\"{a}tty\"{a} and Sierra's work to the endpoint Bergman spaces and obtained necessary and sufficient conditions of the positive measure $\mu$ such that the Toeplitz operator between $L^p_a(\omega)$ and $L_a^1(\omega)$ is bounded and compact for $0<p\leq 1$.
Using some new characterizations on Carleson measures for weighted Bergman spaces, Pau and Zhao presented full characterizations \cite[Theorem 1.2 and Theorem 4.2]{Pau} of the boundedness and compactness of the Toeplitz operators $T_{\mu,\beta}: L_{a}^{p}(\omega_{\alpha})\to L_{a}^{q}(\omega_{\gamma})$, where the parameters $\alpha,\beta,\gamma, p$ and $q$ satisfy
\begin{align*}
2+\beta&>\max\{1,p^{-1}\}+p^{-1}(1+\alpha)\\
2+\beta&>\max\{1,q^{-1}\}+q^{-1}(1+\gamma)
\end{align*}
for $0<p, q<\infty$ and $-1<\alpha,\beta, \gamma<\infty$.

Observe that if $p=1$ and $q>1$, Pau and Zhao's result is valid for $\alpha<\beta$. This motivates us to study the questions of Toeplitz operators $T_{\mu,\beta}: L_{a}^{p}(\omega_{\alpha})\to L_{a}^{q}(\omega_{\gamma})$ when the ranges of parameters can not be covered by the ones of Pau and Zhao. In this note, we restrict our attentions to $\beta=\gamma$ partially because of its close relations to the two weight inequalities for Bergman projections \cite{FW}. In this case, Pau and Zhao's conditions read as
\begin{equation}
2+\beta>\max\{1,p^{-1}\}+p^{-1}(1+\alpha)\label{E:PZ01}
\end{equation}
\begin{equation}
2+\beta>\max\{1,q^{-1}\}+q^{-1}(1+\beta).\label{E:PZ02}
\end{equation}
We further note that the condition \eqref{E:PZ02} never hold if the image of Toeplitz operator $T_{\mu,\beta}$ is in $L_a^1(\omega_\beta)$. The first result of this paper is devoted to the boundedness and compactness of $T_{\mu,\beta}:L_{a}^{p}(\omega_{\alpha})\rightarrow L_{a}^{1}(\omega_{\beta})$ for $0<p\leq 1$ and $-1<\alpha,\beta<+\infty$.

\vskip 0.1in

For $0<p,q<\infty$ and $-1<\alpha<+\infty$, a finite positive Borel measure $\mu$ on $\mathbb{D}$ is said to be a \textit{$q$-Carleson measure for} $L_{a}^{p}(\omega_\alpha)$ if there exists a positive constant $C$ such that
\[
\|f\|_{L^q(\mu)}\leq C\|f\|_{L^p(\omega_\alpha)}
\]
for any $f\in L_a^p(\omega_\alpha)$. The measure $\mu$ is a \textit{vanishing $q$-Carleson measure for} $L_a^p(\omega_\alpha)$ \cite[page 166]{zhu2007} if
for any bounded sequence $\{f_{n}\}_{n=1}^{\infty}$ in $L_{a}^{p}(\omega_{\alpha})$ such that $\{f_{n}\}_{n=1}^{\infty}$ uniformly converges to $0$ on each compact subset of $\mathbb{D}$, we have
\[
\lim_{n\to+\infty}\int_{\mathbb{D}}|f_n(z)|^qd\mu(z)=0.
\]

For any $\alpha>0$, let $\mathcal{LB}^\alpha$ be the Banach space of analytic functions $f$ on $\mathbb{D}$ such that
\[
\|f\|_{\mathcal{LB}^{\alpha}}:=|f(0)|+\sup_{z\in\mathbb{D}}(1-|z|^{2})^{\alpha}
\bigg(\log\frac{2}{1-|z|^{2}}\bigg)|f'(z)|<+\infty
\]
and $\mathcal{LB}_{0}^\alpha$ be the set of analytic functions $f$ on $\mathbb{D}$ such that
\[
\lim_{|z|\rightarrow1^{-}}(1-|z|^{2})^{\alpha}\bigg(\log\frac{2}{1-|z|^{2}}\bigg)|f'(z)|=0.
\]

\vskip 0.1in

\begin{theorem}\label{P:standard2}
Let $\mu$ be a positive Borel measure, $0<p\leq1$ and $-1<\alpha,\beta<\infty$. Then
$T_{\mu,\beta}:L_{a}^{p}(\omega_{\alpha})\rightarrow L_{a}^{1}(\omega_{\beta})$
is bounded if and only if $\mu$ is a $\frac{1}{p}$-Carleson measure for $L_{a}^{1}(\omega_{\alpha})$ and $T_{\mu,\frac{2+\alpha}{p}-2}(1)\in \mathcal{LB}^{1}$.
\end{theorem}

\vskip 0.1in

\begin{theorem} \label{T:compact1} Let $\mu$ be a positive Borel measure, $0<p\leq1$ and $-1<\alpha,\beta<\infty$. Then $T_{\mu,\beta}: L^{p}_{a}(\omega_{\alpha})\rightarrow L^{1}_{a}(\omega_{\beta})$ is compact if and only if $\mu$ is a vanishing $\frac{1}{p}$-Carleson measure for $L^{1}_{a}(\omega_{\alpha})$ and $T_{\mu,\frac{2+\alpha}{p}-2}(1)\in\mathcal{LB}_{0}^1$.
\end{theorem}

\vskip 0.1in

As we have mentioned above, when the Toeplitz operator $T_{\mu,\beta}$ acts on the Bergman space $L_a^1(\omega_\alpha)$, Pau and Zhao's characterization can be applied for $\beta>\alpha$. We shall consider the ranges of $0<p\leq1<q<\infty$ and $-1<\beta\leq\alpha<\infty$ and obtain the following results.

\begin{theorem}\label{T:standard:two}
Let $\mu$ be a positive Borel measure, $0<p\leq1<q<\infty$ and $-1<\beta\leq\alpha<\infty$. Then
$T_{\mu,\beta}:L_{a}^{p}(\omega_{\alpha})\rightarrow L_{a}^{q}(\omega_{\beta})$
is bounded if and only if $\mu$ is a $(\frac{1}{p}+\frac{t}{q'})$-Carleson measure for $L_{a}^{1}(\omega_{\alpha})$, where $t=\frac{\beta+2}{\alpha+2}$ and $\frac{1}{q}+\frac{1}{q'}=1$.
\end{theorem}

\vskip 0.1in

\begin{theorem}\label{T:compact standard}
Let $\mu$ be a positive Borel measure, $0<p\leq1<q<\infty$ and $-1<\beta\leq\alpha<\infty$. Then
$T_{\mu,\beta}:L_{a}^{p}(\omega_{\alpha})\rightarrow L_{a}^{q}(\omega_{\beta})$
is compact if and only if $\mu$ is a vanishing $(\frac{1}{p}+\frac{t}{q'})$-Carleson measure for $L_{a}^{1}(\omega_{\alpha})$, where $t=\frac{\beta+2}{\alpha+2}$ and $\frac{1}{q}+\frac{1}{q'}=1$.
\end{theorem}

\vskip 0.1in

The rest of this paper is organized as follows. In section 2,
we prove Theorem \ref{P:standard2} and Theorem \ref{T:standard:two} which provide characterizations for the boundedness of $T_{\mu,\beta}:L_{a}^{p}(\omega_{\alpha})\rightarrow L_{a}^{q}(\omega_{\beta})$ when $0<p\leq q=1$, $-1<\alpha,\beta<\infty$ and $0<p\leq1<q<\infty$, $-1<\beta\leq\alpha<\infty$, respectively.
In section 3, we first discuss the definition of compact Toeplitz operators on $L^p_a(\omega_\alpha)$ for $0<p<1$ and $p\geq 1$. Then we prove Theorem \ref{T:compact1} and Theorem \ref{T:compact standard} to give descriptions for the compactness of $T_{\mu,\beta}:L_{a}^{p}(\omega_{\alpha})\rightarrow L_{a}^{q}(\omega_{\beta})$ under the same assumptions as the boundedness ones.

\vskip 0.1in

In this paper, we will use $a\lesssim b$ to represent that there exists a constant $C=C(\cdot)>0$ satisfying $a\leq Cb$, where the constant $C(\cdot)$ depends on the parameters indicated in the parenthesis, varying under different cases. Moreover, if $a\lesssim b$ and $b\gtrsim a$, then we denote by $a\asymp b$.

\vskip 0.1in

\section{Boundedness of Toeplitz operators}
The goal of this section is to prove Theorem \ref{P:standard2} and Theorem \ref{T:standard:two}. We start with the following result (see \cite[Theorem 5.3]{zhu2007}) which describes the dual space of $L_{a}^{1}(\omega_{\alpha})$ for $-1<\alpha<+\infty$ and is frequently used throughout our work.

\begin{lemma}\label{dual bloch}
Let $-1<\alpha<+\infty$. Then the dual space of $L_{a}^{1}(\omega_{\alpha})$ is the Bloch space $\mathcal{B}$ via the integral pairing,
\[
\langle f,g\rangle_{L^{2}(\omega_{\alpha})}=\lim_{r\rightarrow1^{-}}\int_{\mathbb{D}}f(rz)\overline{g(rz)}\omega_{\alpha}(z)dA(z),
~~f\in L_{a}^{1}(\omega_{\alpha}),~g\in\mathcal{B},
\]
where the Bloch space $\mathcal{B}$ is the Banach space of analytic functions such that
\[
\|f\|_{\mathcal{B}}:=|f(0)|+\sup_{z\in\mathbb{D}}(1-|z|^2)|f'(z)|<+\infty.
\]
\end{lemma}

\vskip 0.1in

For $z\in\mathbb{D}$ and $r\in(0,1)$, we will use $B(z,r)=\{\zeta\in\mathbb{D}:|z-\zeta|<r\}$ and $D(z,r)=\{\zeta\in\mathbb{D}:\rho(z,\zeta)<r\}$ to represent the Euclidean disk and the pseudohyperbolic disk respectively, where $\rho(z,\zeta)=\big|\frac{z-\zeta}{1-\bar{z}\zeta}\big|$.

\begin{proposition}\label{P:carleson} Let $\mu$ be a positive Borel measure, $0<p\leq 1$ and $-1<\alpha,\beta<\infty$. If
$T_{\mu,\beta}:L_{a}^{p}(\omega_{\alpha})\rightarrow L_{a}^{1}(\omega_{\beta})$ is bounded, then $\mu$ is a $\frac{1}{p}$-Carleson measure for $L_a^1(\omega_\alpha)$.
\end{proposition}
\begin{proof}
Let $c\in (\frac{1}{p}, +\infty)$ and $z\in\mathbb{D}$, consider the functions
\[
f_{z}(\zeta)=\frac{(1-|z|^{2})^{(\alpha+2)(c-\frac{1}{p})}}{(1-\bar{z}\zeta)^{\alpha c+2c}}
\]
and
\[ g_{z}(\zeta)=\frac{(1-|z|^{2})^{\alpha c+2c}}{(1-\bar{z}\zeta)^{\alpha c+2c}},~\zeta\in\mathbb{D}.
\]
From \cite[Lemma 3.10]{zhu2007} and direct calculations, we have
\[
\|f_z\|_{L^{p}(\omega_{\alpha})}\asymp1
\]
and
\[
\|g_z\|_{\mathcal{B}}\asymp1.
\]
Let $r\in(0,1)$. By Lemma \ref{dual bloch} and the boundedness of $T_{\mu,\beta}$,
\begin{align}\label{f_z}
\langle T_{\mu,\beta}(f_z),g_z\rangle_{L^{2}(\omega_{\beta})}
&=\int_{\mathbb{D}}f_{z}(\zeta)\overline{g_{z}(\zeta)}d\mu(\zeta)\nonumber\\
&=(1-|z|^{2})^{(\alpha+2)(2c-\frac{1}{p})}\int_{\mathbb{D}}\frac{1}{|1-\bar{z}\zeta|^{2\alpha c+4c}}d\mu(\zeta)\nonumber\\
&\geq (1-|z|^{2})^{(\alpha+2)(2c-\frac{1}{p})}\int_{D(z,r)}\frac{1}{|1-\bar{z}\zeta|^{2\alpha c+4c}}d\mu(\zeta)\nonumber\\
&\asymp\frac{\mu(D(z,r))}{(1-|z|^{2})^{\frac{\alpha+2}{p}}}
\end{align}
and
\[
\langle T_{\mu,\beta}(f_z),g_z\rangle_{L^{2}(\omega_{\beta})}\leq \|T_{\mu,\beta}(f_z)\|_{L^{1}(\omega_{\beta})}\|g_z\|_{\mathcal{B}}\leq \|T_{\mu,\beta}\|_{L^p(\omega_{\alpha})\to L^1(\omega_{\beta})}\|f_z\|_{L^p(\omega_{\alpha})}\|g_z\|_{\mathcal{B}}.
\]
Therefore,
\[
\sup_{z\in\mathbb{D}}\frac{\mu(D(z,r))}{(1-|z|^{2})^{\frac{\alpha+2}{p}}}<+\infty.
\]
It follows that $\mu$ is a $\frac{1}{p}$-Carleson measure for $L_{a}^{1}(\omega_{\alpha})$.
\end{proof}

Before stating the proof of Theorem \ref{P:standard2}, we give a useful estimation of Bloch functions.

\begin{lemma}\label{P:bloch} Let $0<p\leq1$, $-1<\alpha<\infty$ and $t=\frac{\alpha+2}{p}+1$. If $\mu$ is a $\frac{1}{p}$-Carleson measure for $L_{a}^{1}(\omega_{\alpha})$, then for any $g\in\mathcal{B}$,
\[
\sup_{\zeta\in\mathbb{D}}(1-|\zeta|^{2})\int_{\mathbb{D}}\frac{|g(z)-g(\zeta)|}{|1-z\bar{\zeta}|
^t}d\mu(z)
\lesssim\|g\|_{\mathcal{B}}.
\]
\end{lemma}

\vskip 0.1in

To prove Lemma \ref{P:bloch}, recall that a sequence $\{z_{j}\}_{j=1}^{\infty}\subseteq\mathbb{D}$ is called \textit{$\delta$-separated} ($\delta>0$) if $\rho(z_{i},z_{k})\geq\delta$ when $i\neq k$. For $r\in(0,1)$, a sequence $\{z_{j}\}_{j=1}^{\infty}$ is said to be an \textit{$r$-lattice} if $\{z_{j}\}_{j=1}^{\infty}$ satisfies
\[\mathbb{D}=\bigcup\limits_{j=1}^{\infty}D(z_{j},5r)
\]
and $\frac{r}{5}-$separated. We will frequently use the fact that every $z\in\mathbb{D}$ belongs to at most $N=N(r)$ pseudohyperbolic disks $D(z_{j},r)$, where $\{z_{j}\}_{j=1}^{\infty}$ is a separated sequence (see Lemma 12 of Chapter 2 in \cite{D1} or Lemma 3 in \cite{Lu}).

\begin{proof}[Proof of Lemma \ref{P:bloch}]
Let $t=\frac{\alpha+2}{p}+1$ and
\[
\lambda\in(\textmd{max}\{0,3-t\},1).
\]
For any $g\in\mathcal{B}$, it follows from \cite[Proposition 2.4]{WL-2010} that
\begin{align}\label{bloch-}
\sup_{z,\zeta\in\mathbb{D},z\neq\zeta}
\frac{|g(z)-g(\zeta)|(1-|z|^{2})^{1-\lambda}(1-|\zeta|^{2})^{1-\lambda}}{|z-\zeta||1-\bar{z}\zeta|^{1-2\lambda}}\lesssim\|g\|_{\mathcal{B}}.
\end{align}
Let $\{\xi_{j}\}_{j=1}^{\infty}$ be an $r$-lattice for a fixed $r\in(0,1)$. Using the property of subharmonicity and the condition that $\mu$ is a $\frac{1}{p}$-Carleson measure for $L_{a}^{1}(\omega_{\alpha})$, for $\zeta\in\mathbb{D}$, we deduce
\begin{align}\label{lattice}
\int_{\mathbb{D}}\frac{(1-|z|^{2})^{\lambda-1}}{|1-\bar{z}\zeta|^{2\lambda+t-2}}d\mu(z)
\leq&\sum_{j=1}^{\infty}\mu(D(\xi_{j},5r))
\sup_{z\in D(\xi_{j},5r)}\frac{(1-|z|^{2})^{\lambda-1}}{|1-\bar{z}\zeta|^{2\lambda+t-2}}\nonumber\\
\leq&\sum_{j=1}^{\infty}\mu(D(\xi_{j},5r))
\sup_{z\in D(\xi_{j},5r)}\frac{(1-|z|^{2})^{\lambda-1}}{(\frac{r(1-|z|^{2})}{2(1+r)})^{2}}\int_{B(z,\frac{r(1-|z|^{2})}{2(1+r)})}
\frac{1}{|1-\bar{u}\zeta|^{2\lambda+t-2}}dA(u)\nonumber\\
\lesssim&\sum_{j=1}^{\infty}\mu(D(\xi_{j},5r))
\sup_{z\in D(\xi_{j},5r)}\frac{1}{(1-|z|^{2})^{\alpha+2}}
\int_{D(z,r)}\frac{(1-|u|^{2})^{\lambda+\alpha-1}}{|1-\bar{u}\zeta|^{2\lambda+t-2}}dA(u)\nonumber\\
\lesssim&\sum_{j=1}^{\infty}\frac{\mu(D(\xi_{j},5r))}{(1-|\xi_{j}|^{2})^{\alpha+2}}
\int_{D(\xi_{j},6r)}\frac{(1-|u|^{2})^{\lambda+\alpha-1}}{|1-\bar{u}\zeta|^{2\lambda+t-2}}dA(u)\nonumber\\
\asymp&\sum_{j=1}^{\infty}\frac{\mu(D(\xi_{j},5r))}{(1-|\xi_{j}|^{2})^{\frac{\alpha+2}{p}}}
\int_{D(\xi_{j},6r)}\frac{(1-|u|^{2})^{\lambda+\alpha-1+\frac{(2+\alpha)(1-p)}{p}}}{|1-\bar{u}\zeta|^{2\lambda+t-2}}dA(u)\nonumber\\
\lesssim&\int_{\mathbb{D}}\frac{(1-|u|^{2})^{\lambda+t-4}}
{|1-\bar{u}\zeta|^{2\lambda+t-2}}dA(u).
\end{align}
Let
\[
I(\zeta)=(1-|\zeta|^{2})\int_{\mathbb{D}}\frac{|g(z)-g(\zeta)|}{|1-z\bar{\zeta}|^{t}}d\mu(z).
\]
Combining \eqref{bloch-}, \eqref{lattice} with \cite[Lemma 3.10]{zhu2007},
for any $\zeta\in\mathbb{D}$, we obtain
\begin{align*}\label{2s3}
I(\zeta)=&(1-|\zeta|^{2})^{\lambda}
\int_{\mathbb{D}}\frac{|g(z)-g(\zeta)|(1-|z|^{2})^{1-\lambda}(1-|\zeta|^{2})^{1-\lambda}|}{|z-\zeta||1-\bar{z}\zeta|^{1-2\lambda}}
\frac{|z-\zeta|}{|1-\bar{z}\zeta|}\frac{(1-|z|^{2})^{\lambda-1}}{|1-\bar{z}\zeta|^{2\lambda+t-2}}d\mu(z)\nonumber\\
\lesssim&\|g\|_{\mathcal{B}}(1-|\zeta|^{2})^{\lambda}\int_{\mathbb{D}}\frac{(1-|z|^{2})^{\lambda-1}}
{|1-\bar{z}\zeta|^{2\lambda+t-2}}d\mu(z)\nonumber\\
\lesssim&\|g\|_{\mathcal{B}}(1-|\zeta|^{2})^{\lambda}\int_{\mathbb{D}}\frac{(1-|u|^{2})^{\lambda+t-4}}
{|1-\bar{u}\zeta|^{2+(\lambda+t-4)+\lambda}}dA(u)\nonumber\\
\asymp&\|g\|_{\mathcal{B}},
\end{align*}
which completes the proof.
\end{proof}

Now we are going to prove Theorem \ref{P:standard2}. Recall that the Toeplitz operator $T_{\mu,\beta}:L_{a}^{p}(\omega_{\alpha})\rightarrow L_{a}^{q}(\omega_{\gamma})$ is bounded \cite[page 164]{zhu2007} if there exists a constant $C_{p,q}$ such that
for all polynomials $f$
\[
\|T_{\mu,\beta}(f)\|_{L^q(\omega_\gamma)}\leq C_{p,q}\|f\|_{L^p(\omega_\alpha)}.
\]

\begin{proof}[Proof of Theorem \ref{P:standard2}]
Let
$
\eta=\frac{2+\alpha-p}{p}
$
and
$
\omega_{\eta}(z)=(\frac{2+\alpha}{p})(1-|z|^{2})^{\frac{2+\alpha-p}{p}}
$
be a standard weight. From Lemma \ref{dual bloch}, for any $f\in L^p_a(\omega_\alpha)$ and $g\in\mathcal{B}$, it follows that
\begin{eqnarray*}
\langle T_{\mu,\beta}(f), g\rangle_{L^{2}(\omega_{\beta})}=\langle f,g\rangle_{L^{2}(\mu)}
=\int_{\mathbb{D}}P_{\omega_{\eta}}(f\bar{g})(z)d\mu(z)+
\int_{\mathbb{D}}(f\bar{g})(z)d\mu(z)-\int_{\mathbb{D}}P_{\omega_{\eta}}(f\bar{g})(z)d\mu(z),
\end{eqnarray*}
where $P_{\omega_{\eta}}$ is the orthogonal projection from $L^{2}(\omega_{\eta})$ to the weighted Bergman space $L_{a}^{2}(\omega_{\eta})$.
Let
\begin{eqnarray}\label{bound:I1}
I_{1}=\int_{\mathbb{D}}P_{\omega_{\eta}}(f\bar{g})(z)d\mu(z)
\end{eqnarray}
and
\begin{eqnarray}\label{bound:I2}
I_{2}=\int_{\mathbb{D}}(f\bar{g})(z)d\mu(z)-\int_{\mathbb{D}}P_{\omega_{\eta}}(f\bar{g})(z)d\mu(z).
\end{eqnarray}

\vskip 0.1in

Now we are devoted to estimating $I_{1}$. By \cite[Theorem 4.14]{zhu2007},
\begin{eqnarray*}
|I_{1}| &\leq&\int_{\mathbb{D}}|f(\zeta)|\omega_{\eta}(\zeta)|g(\zeta)|
\big|\int_{\mathbb{D}}\frac{1}{(1-\bar{z}\zeta)^{\eta+2}}d\mu(z)\big|dA(\zeta)\\
&=&\int_{\mathbb{D}}|f(\zeta)|^{p}|f(\zeta)|^{1-p}
\omega_{\eta}(\zeta)|g(\zeta)|\big|\int_{\mathbb{D}}\frac{1}{(1-\bar{z}\zeta)^{\eta+2}}d\mu(z)\big|dA(\zeta)\\
&\leq&\int_{\mathbb{D}}|f(\zeta)|^{p}\frac{\|f\|_{L_{a}^{p}(\omega_{\alpha})}^{1-p}}{(1-|\zeta|^{2})^{\frac{(2+\alpha)(1-p)}{p}}}
\omega_{\eta}(\zeta)|g(\zeta)|\big|\int_{\mathbb{D}}\frac{1}{(1-\bar{z}\zeta)^{\eta+2}}d\mu(z)\big|dA(\zeta)\\
&=&\|f\|_{L_{a}^{p}(\omega_{\alpha})}^{1-p}\int_{\mathbb{D}}|f(\zeta)|^{p}(1-|\zeta|^{2})^{\alpha}
(1-|\zeta|^{2})|g(\zeta)|\big|\int_{\mathbb{D}}\frac{1}{(1-\bar{z}\zeta)^{\eta+2}}d\mu(z)\big|dA(\zeta).
\end{eqnarray*}
Denote
\begin{eqnarray*}
\widetilde{I}(\zeta)&\triangleq&(1-|\zeta|^{2})|g(\zeta)|\big|\int_{\mathbb{D}}\frac{1}{(1-\bar{z}\zeta)^{\eta+2}}d\mu(z)\big|\\
&=&(1-|\zeta|^{2})|g(\zeta)|\bigg|\int_{\mathbb{D}}\frac{1-\bar{z}\zeta}{(1-\bar{z}\zeta)^{\eta+2}}
d\mu(z)+\int_{\mathbb{D}}\frac{\bar{z}\zeta}{(1-\bar{z}\zeta)^{\eta+2}}d\mu(z)\bigg|.
\end{eqnarray*}
Observe that $\mu$ is a $\frac{1}{p}$-Carleson measure for $L_{a}^{1}(\omega_{\alpha})$ and $T_{\mu,\eta-1}(1)\in \mathcal{LB}^{1}$. Combining
\cite[Lemma 2.1]{WL-2010} with \cite[Lemma 3.10]{zhu2007}, for $\zeta\in\mathbb{D}$,
\begin{eqnarray*}
\widetilde{I}(\zeta)
&\leq &(1-|\zeta|^{2})|g(\zeta)|\int_{\mathbb{D}}\frac{1}{|1-\bar{z}\zeta|^{\eta+1}}d\mu(z)+
(1-|\zeta|^{2})|g(\zeta)||T_{\mu,\eta-1}(1)'(\zeta)|\\
&\lesssim&(1-|\zeta|^{2})\log\frac{2}{1-|\zeta|^{2}}\|g\|_{\mathcal{B}}
\bigg[\big(\int_{\mathbb{D}}\frac{(1-|z|^{2})^{\alpha}}{|1-z\bar{\zeta}|^{\alpha+2}}dA(z)\big)^{\frac{1}{p}}
+|T_{\mu,\eta-1}(1)'(\zeta)|\bigg]\\
&\lesssim&\|g\|_{\mathcal{B}}+\|g\|_{\mathcal{B}}\|T_{\mu,\eta-1}(1)\|_{\mathcal{LB}^{1}}
\lesssim \|g\|_{\mathcal{B}}.
\end{eqnarray*}
Thus we have
\begin{eqnarray}\label{p:I1}
|I_{1}|\lesssim \|f\|_{L_{a}^{p}(\omega_{\alpha})}\|g\|_{\mathcal{B}}.
\end{eqnarray}

\vskip 0.1in
Next we turn to estimate $I_{2}$. By \cite[Proposition 4.17]{zhu2007}, we obtain
\begin{eqnarray}\label{property-}
L_{a}^{p}(\omega_{\alpha})\subseteq L_{a}^{1}(\omega_{\eta-1})\subseteq L_{a}^{1}(\omega_{\eta}).
\end{eqnarray}
According to \eqref{property-}, \cite[Theorem 4.14]{zhu2007} and Lemma \ref{P:bloch}, we deduce
\begin{align*}
|I_{2}|=&\bigg|\int_{\mathbb{D}}\int_{\mathbb{D}}\frac{f(\zeta)}{(1-z\bar{\zeta})^{\eta+2}}\omega_{\eta}(\zeta)dA(\zeta)
\overline{g(z)}d\mu(z)-\int_{\mathbb{D}}\int_{\mathbb{D}}\frac{f(\zeta)\overline{g(\zeta)}}{(1-z\bar{\zeta})^{\eta+2}}
\omega_{\eta}(\zeta)dA(\zeta)d\mu(z)\bigg|\nonumber\\
\leq&\int_{\mathbb{D}}|f(\zeta)|^{p}|f(\zeta)|^{1-p}\omega_{\eta}(\zeta)
\int_{\mathbb{D}}\frac{|g(z)-g(\zeta)|}{|1-z\bar{\zeta}|^{\eta+2}}d\mu(z)dA(\zeta)\nonumber\\
\leq&\int_{\mathbb{D}}|f(\zeta)|^{p}\frac{\|f\|_{L_{a}^{p}(\omega_{\alpha})}^{1-p}}{(1-|\zeta|^{2})^{\frac{(2+\alpha)(1-p)}{p}}}
\omega_{\eta}(\zeta)\int_{\mathbb{D}}\frac{|g(z)-g(\zeta)|}{|1-z\bar{\zeta}|^{\eta+2}}d\mu(z)dA(\zeta)\nonumber\\
=&\|f\|_{L_{a}^{p}(\omega_{\alpha})}^{1-p}\int_{\mathbb{D}}|f(\zeta)|^{p}(1-|\zeta|^{2})^{\alpha}(1-|\zeta|^{2})
\int_{\mathbb{D}}\frac{|g(z)-g(\zeta)|}{|1-z\bar{\zeta}|^{\eta+2}}d\mu(z)dA(\zeta)\nonumber\\
\lesssim& \|f\|_{L_{a}^{p}(\omega_{\alpha})}\|g\|_{\mathcal{B}}.
\end{align*}
Then the sufficiency part follows from \eqref{bound:I1}, \eqref{bound:I2} and \eqref{p:I1}.

\vskip 0.1in

To show the necessary part, assume that $T_{\mu,\beta}:L_{a}^{p}(\omega_{\alpha})\rightarrow L_{a}^{1}(\omega_{\beta})$ is bounded for $0<p\leq1$. By Proposition \ref{P:carleson}, it remains to show that \[
T_{\mu,\eta-1}(1)\in \mathcal{LB}^{1}.
\]
For $z\in\mathbb{D}$, consider the test functions
\[
h_{z}(\zeta)\triangleq\frac{(1-|z|^{2})}{(1-\zeta\bar{z})^{\eta+2}},~\zeta\in\mathbb{D}.
\]
Note that
$\|h_{z}\|_{L_{a}^{p}(\omega_{\alpha})}\asymp 1$ then
\begin{align}\label{2s5}
|\langle T_{\mu,\beta}(h_{z}),g\rangle_{L^{2}(\omega_{\beta})}|
=(1-|z|^{2})\bigg|\int_{\mathbb{D}}\frac{\overline{g(u)}}{(1-u\bar{z})^{\eta+2}}d\mu(u)\bigg|
\lesssim \|g\|_{\mathcal{B}}.
\end{align}
for any $g\in\mathcal{B}$. Using the following identity,
\begin{align}\label{2s6}
(1-|z|^{2})\overline{g(z)T_{\mu,\eta}(1)(z)}=\langle T_{\mu,\beta}(h_{z}),g\rangle_{L^{2}(\omega_{\beta})}
+(1-|z|^{2})\int_{\mathbb{D}}\frac{\overline{g(z)}-\overline{g(u)}}{(1-u\bar{z})^{\eta+2}}d\mu(u),
\end{align}
Lemma \ref{P:bloch} and (\ref{2s5}), we obtain that for any $g\in\mathcal{B}$,
\begin{eqnarray}\label{2s8}
(1-|z|^{2})|g(z)T_{\mu,\eta}(1)(z)|\lesssim \|g\|_{\mathcal{B}}.
\end{eqnarray}
Take the supremum over all $g\in\mathcal{B}$ with $\|g\|_\mathcal{B}\leq1$ and $g(0)=0$ to the both sides of \eqref{2s8}. According to \cite[Theorem 5.7]{zhu2007}, we have
\begin{eqnarray}\label{2s9}
(1-|z|^{2})\log\frac{2}{1-|z|^{2}}|T_{\mu,\eta}(1)(z)|\lesssim 1.
\end{eqnarray}
Since $\mu$ is a $\frac{1}{p}$-Carleson measure for $L_{a}^{1}(\omega_{\alpha})$ and \cite[Lemma 3.10]{zhu2007}, we have
\begin{align}\label{2s10}
(1-|z|^{2})\log\frac{2}{1-|z|^{2}}|T_{\mu,\eta-1}(1)(z)|\leq&
(1-|z|^{2})\log\frac{2}{1-|z|^{2}}\int_{\mathbb{D}}\frac{1}{|1-u\bar{z}|^{\eta+1}}d\mu(u)\nonumber\\
\lesssim&(1-|z|^{2})\log\frac{2}{1-|z|^{2}}
\big(\int_{\mathbb{D}}\frac{(1-|u|^{2})^{\alpha}}{|1-u\bar{z}|^{\alpha+2}}dA(u)\big)^{\frac{1}{p}}
\nonumber\\
\lesssim&1.
\end{align}
Notice that
\begin{eqnarray}\label{2s11}
(\eta+1)T_{\mu,\eta}(1)(z)=(\eta+1)T_{\mu,\eta-1}(1)(z)
+z(T_{\mu,\eta-1}(1))'(z).
\end{eqnarray}
Combining (\ref{2s9}), (\ref{2s10}) with (\ref{2s11}), we obtain $T_{\mu,\eta-1}(1)\in \mathcal{LB}^{1}$, and complete proofs of the necessity part.
\end{proof}

\vskip 0.1in
In the following, we give the proof of Theorem \ref{T:standard:two}.
\begin{proof}[Proof of Theorem \ref{T:standard:two}]
Necessity. Let $c\in(\frac{1}{p},\infty)$, $-1<\beta\leq\alpha<\infty$ and $t=\frac{\beta+2}{\alpha+2}$. For $z\in\mathbb{D}$, take testing functions on $\mathbb{D}$
\[
f_{z}(\zeta)=\frac{(1-|z|^{2})^{(\alpha+2)(c-\frac{1}{p})}}{(1-\bar{z}\zeta)^{\alpha c+2c}} \text{ and }
g_{z}(\zeta)=\frac{(1-|z|^{2})^{(\alpha+2)(c-\frac{t}{q'})}}{(1-\bar{z}\zeta)^{\alpha c+2c}}.
\]
By \cite[Lemma 3.10]{zhu2007} and direction calculations,
\[\|f_{z}\|_{L_{a}^{p}(\omega_{\alpha})}\asymp1
\]
and
\[
\|g_{z}\|_{L_{a}^{q'}(\omega_{\beta})}\asymp1.
\]
It follows from $\textmd{H}\ddot{\textmd{o}}\textmd{lder's}$ inequality and the boundedness of $T_{\mu,\beta}$ that
\[
\langle T_{\mu,\beta}(f_z),g_z\rangle_{L^{2}(\omega_{\beta})}\leq \|T_{\mu,\beta}(f_z)\|_{L^{q}(\omega_{\beta})}\|g_z\|_{L^{q'}(\omega_{\beta})}\leq \|T_{\mu,\beta}\|_{L^p(\omega_{\alpha})\to L^q(\omega_{\beta})}\|f_z\|_{L^p(\omega_{\alpha})}\|g_z\|_{L^{q'}(\omega_{\beta})}.
\]
Take $r\in(0,1)$. According to Fubini's Theorem, we deduce
\begin{align}\label{ff_z}
\langle T_{\mu,\beta}(f_z),g_z\rangle_{L^{2}(\omega_{\beta})}
&=\int_{\mathbb{D}}f_{z}(\zeta)\overline{g_{z}(\zeta)}d\mu(\zeta)\nonumber\\
&=(1-|z|^{2})^{(\alpha+2)(2c-\frac{1}{p}-\frac{t}{q'})}
\int_{\mathbb{D}}\frac{1}{|1-\bar{z}\zeta|^{2\alpha c+4c}}d\mu(\zeta)\nonumber\\
&\geq (1-|z|^{2})^{(\alpha+2)(2c-\frac{1}{p}-\frac{t}{q'})}
\int_{D(z,r)}\frac{1}{|1-\bar{z}\zeta|^{2\alpha c+4c}}d\mu(\zeta)\nonumber\\
&\asymp\frac{\mu(D(z,r))}{(1-|z|^{2})^{(\alpha+2)(\frac{1}{p}+\frac{t}{q'})}}.
\end{align}
Hence,
\[
\sup_{z\in\mathbb{D}}\frac{\mu(D(z,r))}{(1-|z|^{2})^{(\alpha+2)(\frac{1}{p}+\frac{t}{q'})}}<\infty,
\]
which yields that $\mu$ is a $(\frac{1}{p}+\frac{t}{q'})$-Carleson measure for $L_{a}^{1}(\omega_{\alpha})$.

Sufficiency. Notes that $\beta\leq\alpha$, From \cite[Proposition 4.17]{zhu2007}, for any
$h\in L_{a}^{q'}(\omega_{\beta})$, we have
\begin{align}\label{prop-}
\int_{\mathbb{D}}|h(z)|^{\frac{q'}{t}}\omega_{\alpha}(z)dA(z)\lesssim
\big(\int_{\mathbb{D}}|h(z)|^{q'}\omega_{\beta}(z)dA(z)\big)^{\frac{1}{t}}.
\end{align}
Assume that $\mu$ is a $(\frac{1}{p}+\frac{t}{q'})$-Carleson measure for $L_{a}^{1}(\omega_{\alpha})$.
Together with Fubini's Theorem, $\textmd{H}\ddot{\textmd{o}}\textmd{lder's}$ inequality and \eqref{prop-}, for any polynomials $f$ and $g$,
\begin{align}\label{density}
|\langle T_{\mu,\beta}(f),g\rangle_{L^{2}(\omega_{\beta})}|&=|\langle f,g\rangle_{L^{2}(\mu)}|\nonumber\\
&\leq\int_{\mathbb{D}}|f(z)g(z)|d\mu(z)\nonumber\\
&\lesssim\bigg(\int_{\mathbb{D}}|f(z)g(z)|^{\frac{pq'}{q'+pt}}\omega_{\alpha}(z)dA(z)\bigg)^{\frac{q'+pt}{pq'}}\nonumber\\
&\leq \|f\|_{L_{a}^{p}(\omega_{\alpha})}\bigg(\int_{\mathbb{D}}|g(z)|^{\frac{q'}{t}}\omega_{\alpha}(z)dA(z)\bigg)^{\frac{t}{q'}}\nonumber\\
&\lesssim\|f\|_{L_{a}^{p}(\omega_{\alpha})}\|g\|_{L_{a}^{q'}(\omega_{\beta})}.
\end{align}
Therefore, $T_{\mu,\beta}:L_{a}^{p}(\omega_{\alpha})\rightarrow L_{a}^{q}(\omega_{\beta})$ is bounded provided $0<p\leq1<q<\infty$ and $-1<\beta\leq\alpha<\infty$.
\end{proof}

\section{Compactness of Toeplitz operators}

In this section, we proceed to give descriptions for the compactness of Toeplitz operators $T_{\mu,\beta}$ between  $L_{a}^{p}(\omega_{\alpha})$ and $L_{a}^{q}(\omega_{\beta})$, and prove Theorem \ref{T:compact1} and Theorem \ref{T:compact standard}.

\vskip 0.1in

Recall for $1\leq p,q<\infty$, the linear operator $T_{\mu,\beta}:L_a^p(\omega_\alpha)\to L^q_a(\omega_\beta)$ is compact if for any bounded sequence $\{f_{n}\}_{n=1}^{\infty}$ in $L_{a}^{p}(\omega_{\alpha})$, there exists a subsequence of $\{T_{\mu,\beta}(f_{n})\}_{n=1}^{\infty}$ that converges in $L_{a}^{q}(\omega_{\beta})$. Note that $L_{a}^{p}(\omega_{\alpha})$ is not a Banach space for $p\in(0,1)$. When $T_{\mu,\beta}$ acts on the metric space $L_a^p(\omega_\alpha)$ for $p\in(0,1)$, we call $T_{\mu,\beta}$ is compact if
\[
\lim\limits_{n\rightarrow\infty}\|T_{\mu,\beta}(f_{n})\|_{L_{a}^{q}(\omega_{\beta})}=0,
\]
where $\{f_{n}\}_{n=1}^{\infty}$ is a bounded sequence in $L_{a}^{p}(\omega_{\alpha})$ such that $\{f_{n}\}_{n=1}^{\infty}$ uniformly converges to $0$ on any compact subsets of $\mathbb{D}$.

\vskip 0.1in

For the sake of completeness, we state and include a proof of the following known fact, which shows that definitions of the compactness for $1\leq p<\infty$ and $0<p<1$ are consistent.

\begin{proposition}\label{compact claim}
Let $-1<\alpha,\beta<\infty$. Then
$
T_{\mu,\beta}:L_{a}^{1}(\omega_{\alpha})\rightarrow L_{a}^{1}(\omega_{\beta})
$
is compact if and only if
\[
\lim\limits_{j\rightarrow\infty}\|T_{\mu,\beta}(f_{j})\|_{L_{a}^{1}(\omega_{\beta})}=0,
\]
where $\{f_{j}\}_{j=1}^{\infty}$ is a bounded sequence in $L_{a}^{1}(\omega_{\alpha})$ such that $\{f_{j}\}_{j=1}^{\infty}$ uniformly converges to $0$ on any compact subsets of $\mathbb{D}$.
\end{proposition}

\begin{proof}[Proof of Proposition \ref{compact claim}] Sufficiency. Let $\{f_{l}\}_{l=1}^{\infty}$ be any bounded sequence in $L_{a}^{1}(\omega_{\alpha})$. It follows from \cite[Theorem 4.14]{zhu2007} that $\{f_{l}\}_{l=1}^{\infty}$ is uniformly bounded on any compact subsets of $\mathbb{D}$. From Montel's Theorem, there exist a subsequence $\{f_{l_{m}}\}_{m=1}^{\infty}$ of $\{f_{l}\}_{l=1}^{\infty}$ and an analytic function $f$ such that $\{f_{l_{m}}\}_{m=1}^{\infty}$ uniformly converges to $f$ on any compact subsets of $\mathbb{D}$ as $m\rightarrow\infty$. Moreover, $f\in L_{a}^{1}(\omega_{\alpha})$. Since the sequence $\{f_{l_{m}}-f\}_{m=1}^{\infty}$ uniformly converges to $0$ on any compact subsets of $\mathbb{D}$, together with the hypothesis, we obtain $\lim\limits_{m\rightarrow\infty}\|T_{\mu,\beta}(f_{l_{m}}-f)\|_{L_{a}^{1}(\omega_{\beta})}=0$, which yields that $T_{\mu,\beta}:L_{a}^{1}(\omega_{\alpha})\rightarrow L_{a}^{1}(\omega_{\beta})$ is compact.

\vskip 0.1in

We divide the proof of necessity into two steps. Assume that
$T_{\mu,\beta}:L_{a}^{1}(\omega_{\alpha})\rightarrow L_{a}^{1}(\omega_{\beta})$ is compact.

\textbf{Step 1.} We will show that $\mu$ is a vanishing Carleson measure for $L_{a}^{1}(\omega_{\alpha})$. Denote by $S_{\mu,\alpha}$ the integral operator acting on $\mathcal{B}$, for any $g\in\mathcal{B}$,
\[
S_{\mu,\alpha}(g)(z):=\int_{\mathbb{D}}\frac{g(\zeta)}{(1-z\bar{\zeta})^{\alpha+2}}d\mu(\zeta),~ z\in\mathbb{D}.
\]
From Lemma \ref{dual bloch}, for any $f\in L_{a}^{1}(\omega_{\alpha})$ and $g\in\mathcal{B}$, we have
\begin{align}\label{compact-}
\langle f,S_{\mu,\alpha}(g)\rangle_{L^{2}(\omega_{\alpha})}=
\langle T_{\mu,\beta}(f),g\rangle_{L^{2}(\omega_{\beta})}=\langle f,T_{\mu,\beta}^{*}(g)\rangle_{L^{2}(\omega_{\alpha})}.
\end{align}
It follows from \eqref{compact-} and the compactness of $T_{\mu,\beta}$ that $S_{\mu,\alpha}$ is compact on $\mathcal{B}$.
Choose $c\in\mathbb{R}$ such that $c>1$. For $z\in\mathbb{D}$, take testing functions
\[
f_{z}(\zeta)=\frac{(1-|z|^{2})^{(\alpha+2)(c-1)}}{(1-\bar{z}\zeta)^{\alpha c+2c}} \text{ and }
g_{z}(\zeta)=\frac{(1-|z|^{2})^{\alpha c+2c}}{(1-\bar{z}\zeta)^{\alpha c+2c}},~\zeta\in\mathbb{D}.
\]
According to \cite[Lemma 3.10]{zhu2007} and the fact that $\mu$ is a Carleson measure for $L_a^1(\omega_{\alpha})$ which is due to Proposition \ref{P:carleson},
for each $\zeta\in \mathbb{D}$, we get
\begin{align*}
|S_{\mu,\alpha}(g_z)(\zeta)|&
\leq\frac{(1-|z|^{2})^{\alpha c+2c}}{(1-|\zeta|)^{\alpha+2}}\int_{\mathbb{D}}\frac{1}{|1-\bar{z}u|^{\alpha c+2c}}d\mu(u)\\
&\lesssim\frac{(1-|z|^{2})^{\alpha c+2c}}{(1-|\zeta|)^{\alpha+2}}
\int_{\mathbb{D}}\frac{(1-|u|^{2})^{\alpha}}{|1-\bar{z}u|^{\alpha c+2c}}dA(u) \\
&\asymp\big(\frac{1-|z|}{1-|\zeta|}\big)^{\alpha+2},
\end{align*}
which yields that
\begin{align}\label{CCP1}
\lim_{|z|\to 1^-}S_{\mu,\alpha}(g_z)(\zeta)=0.
\end{align}
Note that $\{g_z\}$ is a bounded sequence in $\mathcal{B}$. Then there exists a subsequence $\{g_{z_k}\}_{k=1}^{\infty}$ such that
$\lim\limits_{k\to+\infty}\|S_{\mu,\alpha}(g_{z_k})-g\|_{\mathcal{B}}=0$ by the compactness of $S_{\mu,\alpha}$.
Since norm convergence implies pointwise convergence in $\mathcal{B}$, together with \eqref{CCP1}, we know $g=0$, and hence,
$\lim\limits_{k\to+\infty}\|S_{\mu,\alpha}(g_{z_k})\|_{\mathcal{B}}=0$.

Next we claim that
\begin{align}\label{supplement}
\lim\limits_{|z|\to 1^{-}}\|S_{\mu,\alpha}(g_{z})\|_{\mathcal{B}}=0.
\end{align}
If otherwise, then there exists a subsequence $\{g_{z_n}\}_{n=1}^{\infty}$ such that
\[
\lim\limits_{n\to+\infty}\|S_{\mu,\alpha}(g_{z_n})\|_{\mathcal{B}}=\delta_0>0.
\]
Apply the initial arguments to $\{g_{z_n}\}_{n=1}^{\infty}$, we can find a subsequence $\{g_{z_{n_k}}\}_{k=1}^{\infty}$ such that
$\lim\limits_{k\to+\infty}\|S_{\mu,\alpha}(g_{z_{n_k}})\|_{\mathcal{B}}=0$, which is a contradiction.

Take $r\in(0,1)$. Using a similar argument as in the proof of Proposition \ref{P:carleson}, for any $z\in\mathbb{D}$,
\[
\frac{\mu(D(z,r))}{(1-|z|^{2})^{\alpha+2}}\lesssim \langle T_{\mu,\beta}(f_{z}),g_{z}\rangle_{L^{2}(\omega_{\beta})}=
\langle f_{z},S_{\mu,\alpha}(g_{z})\rangle_{L^{2}(\omega_{\alpha})}
\leq \|f_{z}\|_{L_{a}^{1}(\omega_{\alpha})}\|S_{\mu,\alpha}(g_{z})\|_{\mathcal{B}}.
\]
Combining with \eqref{supplement}, we deduce
\[
\lim_{|z|\rightarrow1^{-}}\frac{\mu(D(z,r))}{(1-|z|^{2})^{\alpha+2}}=0,
\]
which finishes this step.

\textbf{Step 2.} Let $\{f_{j}\}_{j=1}^{\infty}$ be any bounded sequence in $L_{a}^{1}(\omega_{\alpha})$ such that $\{f_{j}\}_{j=1}^{\infty}$ uniformly converges to $0$ on any compact subsets of $\mathbb{D}$. For $\delta\in(0,1)$, consider a $\delta$-lattice $\{\xi_{i}\}_{i=1}^{\infty}$ and rearrange the sequence $\{\xi_{i}\}_{i=1}^{\infty}$ such that $|\xi_{1}|\leq|\xi_{2}|\leq\cdots\rightarrow1^{-}$. From Step 1, for any given $\varepsilon>0$, there exists $k_{0}\in\mathbb{N}$ such that for any $i>k_{0}$, it holds that
\[
\frac{\mu(D(\xi_{i},5\delta))}{(1-|\xi_{i}|^{2})^{\alpha+2}}<\varepsilon.
\]
Observe that there exists $r_{0}\in(0,1)$ such that
$\mathbb{D}\backslash r_{0}\mathbb{D}\subseteq\bigcup_{i=k_{0}+1}^{\infty}D(\xi_{i},5\delta)$. Moreover, there also exists $k'\in\mathbb{N}$ such that for any $j\geq k'$, $|f_{j}|<\varepsilon$ on $D(0,r_{0})$. Together with the property of subharmonicity, for $\zeta\in\mathbb{D}$ and $j\geq k'$, we obtain
\begin{align*}
|T_{\mu,\beta}(f_{j})(\zeta)|
&\leq\int_{D(0,r_{0})}\frac{|f_{j}(u)|}{|1-\zeta\bar{u}|^{\beta+2}}d\mu(u)+
\int_{\mathbb{D}\backslash D(0,r_{0})}\frac{|f_{j}(u)|}{|1-\zeta\bar{u}|^{\beta+2}}d\mu(u)\\
&<\frac{\varepsilon\mu(D(0,r_{0}))}{(1-|\zeta|)^{\beta+2}}+\frac{1}{(1-|\zeta|)^{\beta+2}}
\sum_{i=k_{0}+1}^{\infty}\int_{D(\xi_{i},5\delta)}|f_{j}(u)|d\mu(u)\\
&<\frac{\varepsilon\mu(D(0,r_{0}))}{(1-|\zeta|)^{\beta+2}}+\frac{1}{(1-|\zeta|)^{\beta+2}}
\sum_{i=k_{0}+1}^{\infty}\mu(D(\xi_{i},5\delta))\sup_{u\in D(\xi_{i},5\delta)}|f_{j}(u)|\\
&\lesssim\frac{\varepsilon\mu(D(0,r_{0}))}{(1-|\zeta|)^{\beta+2}}+\frac{1}{(1-|\zeta|)^{\beta+2}}\sum_{i=k_{0}+1}^{\infty}
\frac{\mu(D(\xi_{i},5\delta))}{(1-|\xi_{i}|^{2})^{\alpha+2}}\int_{D(\xi_{i},6\delta)}|f_{j}(u)|\omega_{\alpha}(u)dA(u)\\
&\lesssim\frac{\varepsilon\mu(D(0,r_{0}))}{(1-|\zeta|)^{\beta+2}}
+\frac{\varepsilon\|f_{j}\|_{L_{a}^{1}(\omega_{\alpha})}}{(1-|\zeta|)^{\beta+2}}.
\end{align*}
Thus, for $\zeta\in\mathbb{D}$, we get
\begin{align}\label{3}
\lim_{j\to \infty}T_{\mu,\beta}(f_j)(\zeta)=0.
\end{align}
There exists a subsequence $\{f_{j_k}\}_{k=1}^{\infty}$ of $\{f_j\}_{j=1}^{\infty}$ such that
\[
\lim\limits_{k\to+\infty}\|T_{\mu,\beta}(f_{j_k})-f\|_{L_{a}^{1}(\omega_{\beta})}=0
\]
by the compactness of $T_{\mu,\beta}$.
Then, up to passing to a subsequence if necessary,
\[
\lim\limits_{k\to+\infty}T_{\mu,\beta}(f_{j_k})(\zeta)=f(\zeta),~\zeta\in\mathbb{D},~\textmd{a.e.}
\]
Together with \eqref{3}, we know $\|f\|_{L_{a}^{1}(\omega_{\beta})}=0$. Hence, we have
\[
\lim\limits_{k\to+\infty}\|T_{\mu,\beta}(f_{j_k})\|_{L_{a}^{1}(\omega_{\beta})}=0.
\]
Thus, we conclude that $\lim\limits_{j\to \infty}\|T_{\mu,\beta}(f_{j})\|_{L_{a}^{1}(\omega_{\beta})}=0$ by contradiction as a similar argument with \eqref{supplement}. This completes the proof.
\end{proof}

Next we give a necessary condition for the compactness of
$T_{\mu,\beta}:L_{a}^{p}(\omega_{\alpha})\rightarrow L_{a}^{1}(\omega_{\beta})$ when $0<p\leq 1$ and $-1<\alpha,\beta<\infty$.

\begin{proposition}\label{compact-P:carleson1} Let $\mu$ be a positive Borel measure, $0<p\leq 1$ and $-1<\alpha,\beta<\infty$. If
$T_{\mu,\beta}:L_{a}^{p}(\omega_{\alpha})\rightarrow L_{a}^{1}(\omega_{\beta})$ is compact, then $\mu$ is a vanishing $\frac{1}{p}$-Carleson measure for $L_a^1(\omega_\alpha)$.
\end{proposition}
\begin{proof}
For $z\in\mathbb{D}$, take the constant $c$ and the same testing functions
$f_{z}$ and $g_{z}$ as ones in the proof of Proposition \ref{P:carleson}.
Observe that $\|f_{z}\|_{L_{a}^{p}(\omega_{\alpha})}\asymp1$ and $\{f_{z}\}$ uniformly converges to $0$ on any compact subsets of $\mathbb{D}$ as $|z|\rightarrow1^{-}$. From the compactness of
$T_{\mu,\beta}:L_{a}^{p}(\omega_{\alpha})\rightarrow L_{a}^{1}(\omega_{\beta})$, we deduce
\[
\lim\limits_{|z|\rightarrow1^{-}}\|T_{\mu,\beta}(f_{z})\|_{L_{a}^{1}(\omega_{\beta})}=0.
\]
Together with \eqref{f_z}, for $r\in(0,1)$, we have
\[
\frac{\mu(D(z,r))}{(1-|z|^{2})^{\frac{\alpha+2}{p}}}\lesssim\langle T_{\mu,\beta}(f_z),g_z\rangle_{L^{2}(\omega_{\beta})}\leq
\|T_{\mu,\beta}(f_{z})\|_{L^1(\omega_{\beta})}\|g_{z}\|_{\mathcal{B}}\rightarrow0,~~\textmd{as}~|z|\rightarrow1^{-},
\]
which yields that
\[
\lim_{|z|\rightarrow1^{-}}\frac{\mu(D(z,r))}{(1-|z|^{2})^{\frac{\alpha+2}{p}}}=0.
\]
Hence, we know that $\mu$ is a vanishing $\frac{1}{p}$-Carleson measure for $L_{a}^{1}(\omega_{\alpha})$.
\end{proof}

We are now in a position to give the proof of Theorem \ref{T:compact1}.
\begin{proof}[Proof of Theorem \ref{T:compact1}]
Let
$
\eta=\frac{2+\alpha-p}{p}
$
be the same constant as the one in the proof of Theorem \ref{P:standard2}.
Let $\{f_{j}\}_{j=1}^{\infty}$ be a sequence in $L_{a}^{p}(\omega_{\alpha})$ such that $\|f_{j}\|_{L_{a}^{p}(\omega_{\alpha})}\leq1$ and $\{f_{j}\}_{j=1}^{\infty}$ uniformly converges to $0$ on any compact subsets of $\mathbb{D}$. Then it is sufficient to show that
\[
\lim\limits_{j\rightarrow\infty}\|T_{\mu,\beta}(f_{j})\|_{L_{a}^{1}(\omega_{\beta})}=0.
\]
By Lemma \ref{dual bloch}, for any $g\in\mathcal{B}$, we have
\begin{eqnarray}\label{c:I1 and I2}
\langle T_{\mu,\beta}(f_{j}), g\rangle_{L^{2}(\omega_{\beta})}&=&\int_{\mathbb{D}}f_{j}(z)\overline{g(z)}d\mu(z)\nonumber\\
&=&\int_{\mathbb{D}}P_{\omega_{\eta}}(f_{j}\bar{g})(z)d\mu(z)+
\bigg(\int_{\mathbb{D}}(f_{j}\bar{g})(z)d\mu(z)-\int_{\mathbb{D}}P_{\omega_{\eta}}(f_{j}\bar{g})(z)d\mu(z)\bigg),
\end{eqnarray}
where $P_{\omega_{\eta}}$ is the orthogonal projection from $L^{2}(\omega_{\eta})$ to $L_{a}^{2}(\omega_{\eta})$.
Let
\begin{eqnarray}\label{c:I1}
I_{1,j}=\int_{\mathbb{D}}P_{\omega_{\eta}}(f_{j}\bar{g})(z)d\mu(z)
\end{eqnarray}
and
\begin{eqnarray}\label{c:I2}
I_{2,j}=\int_{\mathbb{D}}(f_{j}\bar{g})(z)d\mu(z)-\int_{\mathbb{D}}P_{\omega_{\eta}}(f_{j}\bar{g})(z)d\mu(z).
\end{eqnarray}
Here we give estimations for $I_{1,j}$ and $I_{2,j}$ respectively. For simplicity, part of the following proof is omitted, which can be referred to the corresponding arguments in the proof of Theorem \ref{P:standard2}. Using the condition that $\mu$ is a vanishing $\frac{1}{p}$-Carleson measure for $L^{1}_{a}(\omega_{\alpha})$ and $T_{\mu,\eta-1}(1)\in\mathcal{LB}_{0}^1$, together with
\cite[Lemma 3.10]{zhu2007}, for any given $\varepsilon>0$, there exists $\lambda\in(0,1)$ such that for any $\zeta\in\mathbb{D}$ and $\lambda\leq|\zeta|<1$, it holds that
\begin{eqnarray}\label{1c:I2}
(1-|\zeta|^{2})\log\frac{2}{1-|\zeta|^{2}}\int_{\mathbb{D}}\frac{1}{|1-\bar{z}\zeta|^{\eta+1}}d\mu(z)<\varepsilon
\end{eqnarray}
and
\begin{eqnarray}\label{2c:I2}
(1-|\zeta|^{2})\log\frac{2}{1-|\zeta|^{2}}|T_{\mu,\eta-1}(1)^{'}(\zeta)|<\varepsilon.
\end{eqnarray}
Moreover, there also exists $k_{0}\in\mathbb{N}$ such that for any $j>k_{0}$, $|f_{j}|^{p}<\varepsilon$ on $\lambda\mathbb{D}$.
Combining \eqref{1c:I2}, \eqref{2c:I2}, \cite[Theorem 4.14]{zhu2007} with \cite[Lemma 2.1]{WL-2010}, for any $j>k_{0}$, we obtain
\begin{align*}
|I_{1,j}|\leq&\|f_{j}\|_{L_{a}^{p}(\omega_{\alpha})}^{1-p}\int_{\mathbb{D}}|f_{j}(\zeta)|^{p}\omega_{\alpha+1}(\zeta)
|g(\zeta)|\big|\int_{\mathbb{D}}\frac{1}{(1-\bar{z}\zeta)^{\eta+2}}d\mu(z)\big|dA(\zeta)\nonumber\\
\lesssim&\|g\|_{\mathcal{B}}\|f_{j}\|_{L_{a}^{p}(\omega_{\alpha})}^{1-p}\bigg[\int_{\lambda\mathbb{D}}|f_{j}(\zeta)|^{p}
\omega_{\alpha+1}(\zeta)\log\frac{2}{1-|\zeta|^{2}}
\int_{\mathbb{D}}\frac{1}{|1-\bar{z}\zeta|^{\eta+1}}d\mu(z)dA(\zeta)\nonumber\\
+&\int_{\lambda\mathbb{D}}|f_{j}(\zeta)|^{p}
\omega_{\alpha+1}(\zeta)\log\frac{2}{1-|\zeta|^{2}}|T_{\mu,\eta-1}(1)^{'}(\zeta)|dA(\zeta)\nonumber\\
+&\int_{\mathbb{D}\backslash\lambda\mathbb{D}}|f_{j}(\zeta)|^{p}\omega_{\alpha}(\zeta)(1-|\zeta|^{2})\log\frac{2}{1-|\zeta|^{2}}
\int_{\mathbb{D}}\frac{1}{|1-\bar{z}\zeta|^{\eta+1}}d\mu(z)dA(\zeta)\nonumber\\
+&\int_{\mathbb{D}\backslash\lambda\mathbb{D}}|f_{j}(\zeta)|^{p}\omega_{\alpha}(\zeta)(1-|\zeta|^{2})\log\frac{2}{1-|\zeta|^{2}}
|T_{\mu,\eta-1}(1)^{'}(\zeta)|dA(\zeta)\bigg]\nonumber\\
\lesssim&\varepsilon\|g\|_{\mathcal{B}}\|f_{j}\|_{L_{a}^{p}(\omega_{\alpha})}^{1-p}\omega_{\alpha+1}(\lambda\mathbb{D})
+2\varepsilon\|g\|_{\mathcal{B}}\|f_{j}\|_{L_{a}^{p}(\omega_{\alpha})},
\end{align*}
which yields that
\begin{equation}\label{RC07}
|I_{1,j}|\lesssim\varepsilon\|g\|_{\mathcal{B}},~\textmd{for}~j>k_{0}.
\end{equation}
For a fixed $r\in(0,1)$, consider an $r$-lattice $\{\xi_{i}\}_{i=1}^{\infty}$. We rearrange the sequence $\{\xi_{i}\}_{i=1}^{\infty}$ such that $|\xi_{1}|\leq|\xi_{2}|\leq\cdots\rightarrow1^{-}$. By using the hypothesis that $\mu$ is a vanishing $\frac{1}{p}$-Carleson measure for $L_{a}^{1}(\omega_{\alpha})$, for the $\varepsilon$ ($\varepsilon$ is first mentioned in \eqref{1c:I2}), there exists $k\in\mathbb{N}$ such that for any $i>k$
\begin{align}\label{RC02}
\frac{\mu(D(\xi_{i},5r))}{(1-|\xi_{i}|^{2})^{\frac{\alpha+2}{p}}}<\varepsilon.
\end{align}
Notice that there exists $r_{0}\in(0,1)$ such that
\begin{align}\label{RC03}
\mathbb{D}\backslash r_{0}\mathbb{D}\subseteq\bigcup_{i=k+1}^{\infty}D(\xi_{i},5r).
\end{align}
Let
\begin{align}\label{NI(z)}
\widehat{I}(\zeta)=(1-|\zeta|^{2})\bigg|\int_{\mathbb{D}}\frac{g(\zeta)-g(z)}{(1-\bar{z}\zeta)^{\eta+2}}d\mu(z)\bigg|.
\end{align}
Take a constant $c$ such that $c\in(\textmd{max}\{0,1-\eta\},1)$.
Note that there exists $\rho\in(0,1)$ such that for any $\zeta\in\mathbb{D}$ and $\rho\leq|\zeta|<1$, it holds that
\begin{align}\label{2NI(z)}
(1-|\zeta|^{2})^{c}<\varepsilon.
\end{align}
From the proof of Lemma \ref{P:bloch}, for $\zeta\in\mathbb{D}$, we have
\begin{align*}
\widehat{I}(\zeta)\lesssim&\|g\|_{\mathcal{B}}(1-|\zeta|^{2})^{c}
\int_{r_{0}\mathbb{D}}\frac{(1-|z|^{2})^{c-1}}{|1-\bar{\zeta}z|^{2c+\eta}}d\mu(z)+
\|g\|_{\mathcal{B}}(1-|\zeta|^{2})^{c}\int_{\mathbb{D}\backslash r_{0}\mathbb{D}}\frac{(1-|z|^{2})^{c-1}}{|1-\bar{\zeta}z|^{2c+\eta}}d\mu(z)
\nonumber\\
\triangleq&\widehat{I}_{1}(\zeta)+\widehat{I}_{2}(\zeta).
\end{align*}
Together with \eqref{RC02}, \eqref{RC03}, \eqref{2NI(z)}, the property of subharmonicity and \cite[Lemma 3.10]{zhu2007}, for $\zeta\in\mathbb{D}$ and $\rho\leq|\zeta|<1$, we deduce
\begin{align*}
\widehat{I}_{1}(\zeta)\leq&\|g\|_{\mathcal{B}}(1-|\zeta|^{2})^{c}\frac{\mu(r_{0}\mathbb{D})}{(1-r_{0})^{c+\eta+1}}
\lesssim\varepsilon\|g\|_{\mathcal{B}}
\end{align*}
and
\begin{align*}
\widehat{I}_{2}(\zeta)\leq&\|g\|_{\mathcal{B}}(1-|\zeta|^{2})^{c}
\sum_{i=k+1}^{\infty}\mu(D(\xi_{i},5r))\sup_{z\in D(\xi_{i},5r)}\frac{(1-|z|^{2})^{c-1}}{|1-\bar{\zeta}z|^{2c+\eta}}\nonumber\\
\lesssim&\|g\|_{\mathcal{B}}(1-|\zeta|^{2})^{c}\sum_{i=k+1}^{\infty}\frac{\mu(D(\xi_{i},5r))}{(1-|\xi_{i}|^{2})^{\alpha+2}}
\int_{D(\xi_{i},6r)}\frac{(1-|u|^{2})^{c+\alpha-1}}{|1-\bar{\zeta}u|^{2c+\eta}}dA(u)\nonumber\\
\asymp&\|g\|_{\mathcal{B}}(1-|\zeta|^{2})^{c}\sum_{i=k+1}^{\infty}\frac{\mu(D(\xi_{i},5r))}{(1-|\xi_{i}|^{2})^{\frac{\alpha+2}{p}}}
\int_{D(\xi_{i},6r)}\frac{(1-|u|^{2})^{c+\eta-2}}{|1-\bar{\zeta}u|^{2c+\eta}}dA(u)\nonumber\\
\lesssim&\varepsilon\|g\|_{\mathcal{B}}
(1-|\zeta|^{2})^{c}\int_{\mathbb{D}}\frac{(1-|u|^{2})^{c+\eta-2}}{|1-\bar{\zeta}u|^{2c+\eta}}dA(u)\nonumber\\
\asymp&\varepsilon\|g\|_{\mathcal{B}}.
\end{align*}
Hence, we have
\begin{align}\label{RCCCC}
\widehat{I}(\zeta)\lesssim\varepsilon\|g\|_{\mathcal{B}},~\textmd{for}~ \zeta\in\mathbb{D}~\textmd{and}~ \rho\leq|\zeta|<1.
\end{align}
There also exists $k'\in\mathbb{N}$ such that for any $j>k'$, $|f_{j}|^{p}<\varepsilon$ on $\rho\mathbb{D}$.
Combining \eqref{property-}, \eqref{c:I2},  \eqref{NI(z)}, \eqref{RCCCC}, Lemma \ref{P:bloch} with \cite[Theorem 4.14]{zhu2007}, for any $j>k'$, we get
\begin{align*}
|I_{2,j}|\leq&\|f_{j}\|_{L_{a}^{p}(\omega_{\alpha})}^{1-p}
\bigg[\int_{\mathbb{D}\backslash \rho\mathbb{D}}|f_{j}(\zeta)|^{p}(1-|\zeta|^{2})^{\alpha}(1-|\zeta|^{2})
\bigg|\int_{\mathbb{D}}\frac{g(z)-g(\zeta)}{(1-\zeta\bar{z})^{\eta+2}}d\mu(z)\bigg|dA(\zeta)\nonumber\\
+&\int_{\rho\mathbb{D}}|f_{j}(\zeta)|^{p}(1-|\zeta|^{2})^{\alpha}(1-|\zeta|^{2})
\bigg|\int_{\mathbb{D}}\frac{g(z)-g(\zeta)}{(1-\zeta\bar{z})^{\eta+2}}d\mu(z)\bigg|dA(\zeta)\bigg]\nonumber\\
\lesssim&\varepsilon\|f_{j}\|_{L_{a}^{p}(\omega_{\alpha})}\|g\|_{\mathcal{B}}
+\varepsilon\omega_{\alpha}(\rho\mathbb{D})\|f_{j}\|_{L_{a}^{p}(\omega_{\alpha})}^{1-p}\|g\|_{\mathcal{B}},
\end{align*}
from which it follows that
\begin{equation}\label{RC14}
|I_{2,j}|\lesssim\varepsilon\|g\|_{\mathcal{B}},~\textmd{for}~j>k'.
\end{equation}
By \eqref{c:I1 and I2}, \eqref{c:I1}, \eqref{c:I2}, \eqref{RC07} and \eqref{RC14}, it yields that $T_{\mu,\beta}:L_{a}^{p}(\omega_{\alpha})\rightarrow L_{a}^{1}(\omega_{\beta})$ is compact, where $p\in(0,1]$ and $\alpha,\beta\in(-1,+\infty)$. Here we finish the sufficiency part.

\vskip 0.1in

In what follows, we consider the necessity part. From Proposition \ref{compact-P:carleson1}, we know that $\mu$ is a vanishing $\frac{1}{p}$-Carleson measure for $L_{a}^{1}(\omega_{\alpha})$.
It is sufficient to show
\[
T_{\mu,\eta-1}(1)\in\mathcal{LB}_{0}^1,
\]
where
\[
\eta=\frac{2+\alpha-p}{p}.
\]
For $z\in\mathbb{D}$, choose the testing functions
$h_{z}$ be the same as in the proof of Theorem \ref{P:standard2}.
Notice that $\|h_{z}\|_{L_{a}^{p}(\omega_{\alpha})}\asymp 1$ and $\{h_{z}\}$ uniformly converges to $0$ on any compact subsets of $\mathbb{D}$ as $z$ approaches to $\mathbb{T}$. Then
$
\lim\limits_{|z|\rightarrow1^{-}}\|T_{\mu,\beta}(h_{z})\|_{L_{a}^{1}(\omega_{\beta})}=0
$
by the compactness of $T_{\mu,\beta}$. Hence, for any $g\in\mathcal{B}$, we get
\begin{align}\label{RC01}
|\langle T_{\mu,\beta}(h_{z}),g\rangle_{L^{2}(\omega_{\beta})}|=
(1-|z|^{2})\bigg|\int_{\mathbb{D}}\frac{\overline{g(u)}}{(1-u\bar{z})^{\eta+2}}d\mu(u)\bigg|
\leq \|T_{\mu,\beta}(h_{z})\|_{L_{a}^{1}(\omega_{\beta})}\|g\|_{\mathcal{B}}\to 0,
\end{align}
as $|z|\rightarrow1^{-}$.
Together with \eqref{2s6}, \eqref{NI(z)}, \eqref{RCCCC}, \eqref{RC01} and \cite[Theorem 5.7]{zhu2007}, we obtain
\begin{align}\label{RC05}
\lim_{|z|\rightarrow1^{-}}(1-|z|^{2})\log\frac{2}{1-|z|^{2}}|T_{\mu,\eta}(1)(z)|=0.
\end{align}
Since $\mu$ is a vanishing $\frac{1}{p}$-Carleson measure for $L_{a}^{1}(\omega_{\alpha})$, combining with \cite[Lemma 3.10]{zhu2007}, we know
\begin{align}\label{RC06}
\lim_{|z|\rightarrow1^{-}}(1-|z|^{2})\log\frac{2}{1-|z|^{2}}|T_{\mu,\eta-1}(1)(z)|=0.
\end{align}
It follows from \eqref{2s11}, \eqref{RC05} and \eqref{RC06} that $T_{\mu,\eta-1}(1)\in\mathcal{LB}_{0}^{1}$.
Thus, we complete the proof.
\end{proof}

Finally, we give the proof of Theorem \ref{T:compact standard}.
\begin{proof}[Proof of Theorem \ref{T:compact standard}]
Necessity. Let $-1<\beta\leq\alpha<\infty$ and $t=\frac{\beta+2}{\alpha+2}$. Choose a constant $c\in(\frac{1}{p},\infty)$. For $z\in\mathbb{D}$, take the same testing functions
$f_{z}$ and $g_{z}$ as ones in the proof of Theorem \ref{T:standard:two}.
Note that
\[
\|f_{z}\|_{L_{a}^{p}(\omega_{\alpha})}\asymp1
\]
and $\{f_{z}\}$ uniformly converges to $0$ on any compact subsets of $\mathbb{D}$ as $|z|\rightarrow1^{-}$. Then we deduce
\[
\lim\limits_{|z|\rightarrow1^{-}}\|T_{\mu,\beta}(f_{z})\|_{L_{a}^{q}(\omega_{\beta})}=0.
\]
Together with \eqref{ff_z} and $\textmd{H}\ddot{\textmd{o}}\textmd{lder's}$ inequality, for $r\in(0,1)$, we deduce
\[
\frac{\mu(D(z,r))}{(1-|z|^{2})^{(\alpha+2)(\frac{1}{p}+\frac{t}{q'})}}\lesssim\langle T_{\mu,\beta}(f_z),g_z\rangle_{L^{2}(\omega_{\beta})}\leq
\|T_{\mu,\beta}(f_{z})\|_{L^q(\omega_{\beta})}\|g_{z}\|_{L^{q'}(\omega_{\beta})}\rightarrow0,~~\textmd{as}~|z|\rightarrow1^{-},
\]
which yields that $\mu$ is a vanishing $(\frac{1}{p}+\frac{t}{q'})$-Carleson measure for $L_{a}^{1}(\omega_{\alpha})$.

Sufficiency. Assume that $\mu$ is a vanishing $(\frac{1}{p}+\frac{t}{q'})$-Carleson measure for $L_{a}^{1}(\omega_{\alpha})$. Let $\delta\in(0,1)$. Take a $\delta$-lattice $\{\xi_{i}\}_{i=1}^{\infty}$ such that $|\xi_{1}|\leq|\xi_{2}|\leq\cdots\rightarrow1^{-}$. Then
for any given $\varepsilon>0$, there exists $k\in\mathbb{N}$ such that
\begin{align}\label{lat1}
\frac{\mu(D(\xi_{i},5\delta))}{(1-|\xi_{i}|^{2})^{(\alpha+2)(\frac{1}{p}+\frac{t}{q'})}}<\varepsilon,
~\textmd{for~any~}i>k.
\end{align}
Moreover, there exists $r_{0}\in(0,1)$ such that
\begin{align}\label{latt1}
\mathbb{D}\backslash D(0,r_{0})\subseteq \bigcup_{i=k+1}^{\infty}D(\xi_{i},5\delta).
\end{align}
Let $\{f_{n}\}_{n=1}^{\infty}$ be a bounded sequence in $L_{a}^{p}(\omega_{\alpha})$ such that $\|f_{n}\|_{L_{a}^{p}(\omega_{\alpha})}\leq1$ and $\{f_{n}\}_{n=1}^{\infty}$ uniformly converges to $0$ on any compact subsets of $\mathbb{D}$. To show $T_{\mu,\beta}:L_{a}^{p}(\omega_{\alpha})\rightarrow L_{a}^{q}(\omega_{\beta}) $ is compact, it is sufficient to prove that
\[
\lim\limits_{n\rightarrow\infty}\|T_{\mu,\beta}(f_{n})\|_{L_{a}^{q}(\omega_{\beta})}=0.
\]
Let
$
\mu_{r_{0}}=\chi_{D(0,r_{0})}\mu
$
and
$
\mu_{\tilde{r}_{0}}=\chi_{\mathbb{D}\backslash D(0,r_{0})}\mu.
$
According to \eqref{lat1}, \eqref{latt1}, the property of subharmonicity and \cite[Proposition 4.17]{zhu2007}, we deduce
\begin{align}\label{v-carleson}
\|f_{n}\|_{L^{1}(\mu_{\tilde{r}_{0}})}
\leq&\sum_{i=k+1}^{\infty}\mu(D(\xi_{i},5\delta))\sup_{z\in D(\xi_{i},5\delta)}|f_{n}(z)|\nonumber\\
\lesssim&\sum_{i=k+1}^{\infty}\frac{\mu(D(\xi_{i},5\delta))}{(1-|\xi_{i}|^{2})^{\alpha+2}}
\int_{D(\xi_{i},6\delta)}|f_{n}(z)|\omega_{\alpha}(z)dA(z)\nonumber\\
\asymp&\sum_{i=k+1}^{\infty}\frac{\mu(D(\xi_{i},5\delta))}{(1-|\xi_{i}|^{2})^{\frac{(\alpha+2)(q'+pt)}{pq'}}}
\int_{D(\xi_{i},6\delta)}|f_{n}(z)|(1-|z|^{2})^{\frac{(\alpha+2)(q'+pt)}{pq'}-2}dA(z)\nonumber\\
\lesssim&\varepsilon\int_{\mathbb{D}}|f_{n}(z)|(1-|z|^{2})^{\frac{(\alpha+2)(q'+pt)}{pq'}-2}dA(z)
\nonumber\\\lesssim&\varepsilon\|f_{n}\|_{L_{a}^{\frac{pq'}{q'+pt}}(\omega_{\alpha})},~\textmd{for}~n\in\mathbb{N}.
\end{align}
Notice that $\mu_{\tilde{r}_{0}}$ is a $(\frac{q'+pt}{pq'})$-Carleson measure for $L_{a}^{1}(\omega_{\alpha})$. Then \eqref{density} also holds by substituting $\mu_{\tilde{r}_{0}}$ for $\mu$. Hence, combining with \eqref{v-carleson}, we conclude that
\begin{align}\label{v-carleson1}
\|T_{\mu_{\tilde{r}_{0}},\beta}\|_{L_{a}^{p}(\omega_{\alpha})\rightarrow L_{a}^{q}(\omega_{\beta})}
\lesssim\|Id\|_{L_{a}^{\frac{pq'}{q'+pt}}(\omega_{\alpha})\rightarrow L^{1}(\mu_{\tilde{r}_{0}})}\lesssim\varepsilon.
\end{align}
Observe that there exists $k'\in\mathbb{N}$ such that $|f_{n}|<\varepsilon$ on $r_{0}\mathbb{D}$ for any $n>k'$. Since $\mu_{r_{0}}$ is also a $(\frac{1}{p}+\frac{t}{q'})$-Carleson measure for $L_{a}^{1}(\omega_{\alpha})$, together with Fubini's Theorem,
$\textmd{H}\ddot{\textmd{o}}\textmd{lder's}$ inequality and \eqref{prop-}, for any $g\in L_{a}^{q'}(\omega_{\beta})$, it yields that
\begin{align}\label{v-carleson2}
|\langle T_{\mu_{r_{0}},\beta}(f_{n}),g\rangle_{L^{2}(\omega_{\beta})}|
\leq&\int_{\mathbb{D}}|f_{n}(z)g(z)|d\mu_{r_{0}}(z)\nonumber\\
\leq&\|f_{n}\|_{L^{\frac{q'+pt}{q'}}(\mu_{r_{0}})}\|g\|_{L^{\frac{q'+pt}{pt}}(\mu_{r_{0}})}\nonumber\\
\lesssim&\varepsilon\big(\int_{\mathbb{D}}|g(z)|^{\frac{q'}{t}}\omega_{\alpha}(z)dA(z)\big)^{\frac{t}{q'}}\nonumber\\
\lesssim&\varepsilon\|g\|_{L^{q'}(\omega_{\beta})},~~\textmd{for~}n>k'.
\end{align}
Combining \eqref{v-carleson1} with \eqref{v-carleson2}, we obtain
\begin{align*}
\|T_{\mu,\beta}(f_{n})\|_{L_{a}^{q}(\omega_{\beta})}
\lesssim&\|T_{\mu_{r_{0}},\beta}(f_{n})\|_{L_{a}^{q}(\omega_{\beta})}+\|T_{\mu_{\tilde{r}_{0}},\beta}(f_{n})\|_{L_{a}^{q}(\omega_{\beta})}
\nonumber\\
\leq&\sup_{\{g:\|g\|_{L_{a}^{q'}(\omega_{\beta})}\leq1\}}|\langle T_{\mu_{r_{0}},\beta}(f_{n}),g\rangle_{L^{2}(\omega_{\beta})}|
+\|T_{\mu_{\tilde{r}_{0}},\beta}\|_{L_{a}^{p}(\omega_{\alpha})\rightarrow L_{a}^{q}(\omega_{\beta})}
\|f_{n}\|_{L_{a}^{p}(\omega_{\alpha})}\nonumber\\
\lesssim&\varepsilon,~\textmd{for~}n>k'.
\end{align*}
Thus, we conclude that $\lim\limits_{n\rightarrow\infty}\|T_{\mu,\beta}(f_{n})\|_{L_{a}^{q}(\omega_{\beta})}=0$,
which completes the proof.
\end{proof}

\section*{Acknowledgements}
This work is partially supported by National Natural Science Foundation of China.
S.WANG thanks Professor Yongjiang Duan for his guidance and continuous encouragement.
Z.WANG is partially supported by Natural Science Basic Research Program of Shaanxi (Program No. 2020JM-278).
He also thanks the support of the School of Mathematics and Statistics, Northeast Normal University for his visit to Changchun.

%\bibliographystyle{plain}
%\bibliography{weightedbergman}

\begin{thebibliography}{10}



\bibitem{DGWW}
Yongjiang Duan, Kunyu Guo, Siyu Wang, Zipeng Wang,
\newblock Toeplitz operators on the weighted Bergman spaces,
\newblock {arXiv:2002.06771} Feb 2020.

\bibitem{D1}
Peter~L. Duren, Alexander Schuster,
\newblock {Bergman Spaces}, {Math. Surveys Monogr.},
\newblock vol. 100, American Mathematical Society, Providence, RI, 2004.

\bibitem{FW}
Xiang Fang, Zipeng Wang,
\newblock Two weight inequalities for the {B}ergman projection with
doubling measures,
\newblock {Taiwanese J. Math.} 19 (3) (2015) 919--926.


\bibitem{Lu}
Daniel~H. Luecking,
\newblock Embedding theorems for spaces of analytic functions via {K}hinchine's
  inequality,
\newblock {Michigan Math. J.} 40 (2) (1993) 333--358.





\bibitem{L1987}
Daniel~H. Luecking,
\newblock Trace ideal criteria for {T}oeplitz operators,
\newblock {J. Funct. Anal.} 73 (2) (1987) 345--368.





\bibitem{Pau}
Jordi Pau, Ruhan Zhao,
\newblock Carleson measures and {T}oeplitz operators for weighted {B}ergman
  spaces on the unit ball,
\newblock {Michigan Math. J.} 64 (4) (2015) 759--796.


\bibitem{PR}
Jos\'{e}~\'{A}ngel Pel\'{a}ez, Jouni R\"{a}tty\"{a},
\newblock Weighted {B}ergman spaces induced by rapidly increasing weights,
\newblock {Mem. Amer. Math. Soc.} 227 (1066) (2014) vi, 124 pp.


\bibitem{PRS}
Jos\'{e}~\'{A}ngel Pel\'{a}ez, Jouni R\"{a}tty\"{a}, Kian Sierra,
\newblock Berezin transform and {T}oeplitz operators on {B}ergman spaces
  induced by regular weights,
\newblock {J. Geom. Anal.} 28 (1) (2018) 656--687.

%\bibitem{MW}
%Santeri Miihkinen, Jordi Pau, Antti Per\"{a}l\"{a}, Maofa Wang,
%\newblock  Volterra type integration operators from {B}ergman spaces to {H}ardy spaces ,
%\newblock {J. Funct. Anal.} (2020), 108564, doi: https://doi.org/10.1016/j.jfa.2020.108564.




\bibitem{WL-2010}
Xiongliang Wang, Taishun Liu,
\newblock Toeplitz operators on {B}loch-type spaces in the unit ball of {${\bf
  C}^n$},
\newblock {J. Math. Anal. Appl.} 368 (2) (2010) 727--735.


\bibitem{WZZ-2006}
Zhijian Wu, Ruhan Zhao, Nina Zorboska,
\newblock Toeplitz operators on {B}loch-type spaces,
\newblock {Proc. Amer. Math. Soc.} 134 (12) (2006) 3531--3542.





\bibitem{Zhu1988}
Kehe Zhu,
\newblock Positive {T}oeplitz operators on weighted {B}ergman spaces of bounded
  symmetric domains,
\newblock {J. Operator Theory} 20 (2) (1988) 329--357.







\bibitem{zhu2007}
Kehe Zhu,
\newblock {Operator Theory in Function Spaces}, second ed., {Math. Surveys Monogr.},
\newblock vol. 138, American Mathematical Society, Providence, RI, 2007.

\end{thebibliography}

\end{document}